\documentclass[10pt]{article}
\usepackage{amsmath,amsthm,amssymb}
\usepackage{mathtext}
\usepackage{mathtools}

\usepackage[T2A]{fontenc}
\usepackage[utf8]{inputenc}
\usepackage[english]{babel}
\usepackage{graphicx}
\usepackage[colorinlistoftodos]{todonotes}
\usepackage{tikz}

\usepackage{hyperref}

\usepackage[
  paper=a4paper,
  textwidth=135mm,
  textheight=194mm,
  centering
]{geometry}

\newtheorem{proposition}{Proposition}
\newtheorem{lemma}{Lemma}
\newtheorem{corollary}{Corollary}

\newtheorem{theorem}{Theorem}

\newcommand{\zn}{\mathbb{Z}/n\mathbb{Z}}

\title{The edge-isoperimetric inequality for powers of cycles}
\author{Kristiyan Vasilev\footnote{Institute of Mathematics and Informatics, Bulgarian Academy of Sciences. The research was performed under the Project i-METODE, No BG05SFPR001-3.004-0012, funded under the procedure “Support for the development of project-based doctoral study” from the Programme “Education 2021--2027”, co-financed by the European Union.
}}
\date{}

\begin{document}

\maketitle


\begin{abstract}
    This note provides a complete solution to a certain version of the edge-isoperimetric problem for powers of a cycle graph. Namely, it shows that the maximum number of edges inside a vertex subset of $C_n^s$ of size $k$ is achieved by a set of $k$ consecutive vertices.
\end{abstract}

\section{Introduction}

Let $G = (V,E)$ be a simple undirected graph. We consider the graph edge-isoperimetric problem in the following form: given a graph $G$ and $k\in \mathbb{N}$, what is the largest number of edges in an induced subgraph of $G$ with $k$ vertices, i.e 
\[
\max_{|U|=k, U\subset V} e(U).
\]
For a $d$-regular graph, a minimizer of this quantity has the smallest graph perimeter (number of edges between a vertex subset and its complement) among $k$-sets. Recently, this question for Johnson graphs was studied by Raigorosdkii and his students~\cite{dubinin2024lower} (see also references therein).
Various results on this and related quantities can be found in~\cite{chung2004}.

The $s$-th power $G^s$ of graph $G$ is a graph that has the same set of vertices, but in which two vertices are adjacent when their distance in $G$ is at most $s$.
This note solves the edge-isoperimetric problem for the powers $C_n^s$ of the cycle graph $C_n$.

\begin{theorem} \label{main}
    If $n$, $k$, and $s$ are positive integers such that $n\geq k$ and $n > s$, then the maximum 
    \[
    \max_{U\subset V(C_n^s), |U|=k} e(U)
    \]
    is attained by any set of $k$ consecutive vertices of $C_n^s$.
\end{theorem}

Let $n$, $k$, and $s$ be positive integers. We are interested in the maximal number of edges in an induced subgraph of $C_n^s$ with $k$ vertices. Note that if $s+1\geq n/2$, then $C_n^s \cong K_n$, so without loss of generality we may assume throughout this paper that $s<n/2-1$. We may also assume that $k\geq s+2$ as otherwise any set of $k$  consecutive vertices is optimal because all pairs of vertices are adjacent in $C_n^s$.

Note that Theorem~\ref{main} does not provide the classification of maximizers, which remains an open question. The example $U = \{1,3,5\}$ for $n = 6$, $s = 2$ and $k = 3$ shows that the complete classification may be complicated.

\paragraph{Notation.} From now on we will abuse notation slightly and we will associate vertex set of $C_n^s$ with the elements of $\mathbb{Z}/n\mathbb{Z}$, which we will also identify with the elements of $[n]$ by their natural correspondence. We also define the following distances on $[n]$: the counterclockwise, clockwise, and cyclic distances $d_{+}(i,j):=|i-j|$, $d_{-}(i,j) = n-|i-j|$, and $d=\min(d_{+}, d_{-})$. For a vertex $u\in \mathbb{Z}/n\mathbb{Z}$ and a subset $U\subset \mathbb{Z}/n\mathbb{Z}$ such that $u\in U$, we denote the set of the clockwise neighbors of $u$ in the set $U$ by 
\[
N^U_{+}(u):=\left\{v\in U, v\neq u: d_{+}(u, v)\leq s\right \}.
\]
Similarly, 
\[
N^U_{-}(u):=\left\{ v\in U, v\neq u: d_{-}(u,v)\leq s\right \}
\]
and 
\[
N^U(u):=N^U_{-}(u)\cup N^U_{+}(u).
\]
Note that since $n\geq 2s+1$, we have $N^U_{-}(u)\cap N^U_{+}(u) = \varnothing$ for all $u\in C_n^s$ and all $U\subset C_n^s$. 

\section{Application of general bounds}

\subsection{Bound from Tur{\'a}n's theorem}

A classical method for finding the maximal number of edges in a subgraph is by using the celebrated Tur{\'a}n's theorem. 

\begin{theorem}[Tur{\'a}n's] \label{th:turan}
    Let $G$ be a graph with clique number $\omega(G)$ and let $k > \omega(G)$. Then 
    \[
    \max_{U\subset V, |U|=k} e(U) \leq \binom{k}{2} - \frac{\omega(G)}{2} \left \lfloor \frac{k}{\omega(G)} \right \rfloor \left (\left \lfloor \frac{k}{\omega(G)}\right \rfloor - 1 \right).
    \]
\end{theorem}

\begin{lemma}
For $2s + 1 \leq n$ one has $\omega(C_n^s) = s+1$.
\end{lemma}

\begin{proof}
    Let $K$ be a clique in $C_n^s$ and let $u\in K$. Then all vertices of $K$ are among the $2s$ neighbors of $u$ in $C_n^s$.  Let $N_{-}(u):=\{u_1, u_2, \ldots, u_s\}$ and $N_{+}(u):=\{v_1, v_2, \ldots, v_s\}$, where $d_{-}(u, u_i)=i$ and $d_{+}(u, v_i)=i$. Then note that for $i=1,2,\ldots, s$, we have  $d(u_{s+1-i}, v_i)=s+1$, so at most one of $u_{s+1-i}$ and $v_i$ is in $K$. This means that $|K|\leq s+1$. 

On the other hand,  any $s+1$ consecutive vertices of $C_n^s$ form a clique, so indeed $\omega(C_n^s)=s+1$.
\end{proof} 

Now the application of Theorem~\ref{th:turan} to $C_n^s$ gives the following bound: 

\begin{corollary}
For $k\geq s+2$ we have
\[
        \max_{U\subset V, |U|=k} e(U) \leq \binom{k}{2} - \frac{s+1}{2} \left \lfloor \frac{k}{s+1} \right \rfloor \left (\left \lfloor \frac{k}{s+1}\right \rfloor - 1 \right).
\]    
\end{corollary}

This gives 
\[
      \max_{U\subset V, |U|=k} e(U) \leq \frac{sk^2}{2(s+1)} + \mathcal{O}(k).
\]

\subsection{Spectral bound}

Another classical method for bounding the number of edges in a subgraph of a graph $G$ is by using the spectral decomposition of the adjacency matrix $A$ of $G$. 
Despite the fact that the following machinery is well known~\cite{brouwer2011spectra}, it is complicated to find the bound in a pure form.
We use the following 

\begin{proposition} \label{pr:edgesQF}
    Let $G=(V,E)$ be a graph and $U\subset V$. If $A$ is the adjacency matrix of $G$ and $\chi_U$ is the characteristic vector of $U$, then 
    \[
    2e(U)=\langle A\chi_U,\chi_U\rangle.
    \]
\end{proposition}

\begin{proof}
    Note that $(A\chi_U)_i$ is exactly the number of neighbors of $i$ among the vertices of $U$. Therefore $\langle A\chi_U, \chi_U\rangle$ is equal to the sum of the degrees of the vertices of $U$ with respect to the subgraph of $G$ induced by $U$, which is equal to $2e(U)$.
\end{proof} 

Since $A$ is real and symmetric, it admits an orthonormal eigenbasis $v_1, v_2, \ldots, v_n$. Let $\lambda_1\geq \lambda_2 \geq \ldots \lambda_n$ be the corresponding eigenvalues.

\begin{lemma}
    Let $U\subset V$ and $c_i:=\langle v_i, \chi_U\rangle$ for $i=1,2,\ldots, n$. Then 
\[
        \sum_{i=1}^n c_i^2 = |U|.
\]
\end{lemma}

\begin{proof}
    Note that $\chi_U=\sum_{i=1}^n c_iv_i$. Therefore,
\[
   |U|= \langle \chi_U, \chi_U\rangle = \sum_{i, j=1}^n c_ic_j\langle v_i,v_j\rangle =\sum_{i=1}^n c_i^2.
\]
\end{proof} 

\begin{lemma}
    If $U\subset V$ and $c_i:=\langle v_i, \chi_U \rangle$, we have 
    \[
    2e(U) = \sum_{i=1}^n \lambda_i c_i^2.
    \]
\end{lemma}

\begin{proof}
    By Proposition~\ref{pr:edgesQF} we have 
 \[
 2e(U)=\langle A\chi_U, \chi_U\rangle = \left \langle\sum_{i=1}^n \lambda_i c_i v_i, \sum_{j=1}^n c_iv_i \right \rangle = \sum_{i,j=1}^n \lambda_i c_ic_j\langle v_i, v_j \rangle = \sum_{i=1}^n \lambda_i c_i^2.
 \]
\end{proof} 

\noindent For regular graphs, the largest eigenvalue of the adjacency matrix, as well as its corresponding eigenvector, are known explicitly \cite{brouwer2011spectra}.

\begin{lemma}
    For a $d$-regular graph $G$, we have $\lambda_1 = d$ and $v_1=(1, 1, 1,\ldots, 1)^T$.
\end{lemma}

 We can now apply the lemmas to obtain the following:

\begin{corollary}
    Let $U\subset C_n^s$ with $|U|=k$, then 
    \[
    e(U)\leq sk^2/n + \lambda_2(k/2-k^2/2n).
    \]
\end{corollary}

\begin{proof} We have 
\[
    2e(U) = \sum_{i=1}^n \lambda_i c_i^2  \leq 2s c_1^2 + \lambda_2\left(\sum_{i=2}^n c_i^2\right)\leq  2s c_1^2 + \lambda_2\left(|U|-c_1^2\right) = \frac{2sk^2}{n} + \lambda_2 \left(k - \frac{k^2}{n} \right).
\]
\end{proof}

\subsection{Numerical comparison}
We present a numerical comparison of the general methods with the exact value for various values of $k$ and $s$ when $n=1000$.
\begin{table}[h!]
\centering
\begin{tabular}{|r|r|r|r|r|}
\hline
$k$ & $s$ & Exact maximum & Spectral bound & Tur{\'a}n's bound \\
\hline
 54 &  37 & 1295   & 1980   & 1431   \\
118 &  53 & 4823   & 6149   & 6849   \\
359 &  16 & 5608   & 5737   & 60691  \\
210 & 115 & 17480  & 22511  & 21945  \\
243 & 175 & 27125  & 36369  & 29403  \\
313 & 295 & 48675  & 61627  & 48828  \\
433 & 196 & 65562  & 73512  & 93331  \\
404 & 372 & 80910  & 88116  & 81406  \\
439 & 384 & 94656  & 99895  & 96141  \\
473 & 462 & 111573 & 112499 & 111628 \\
\hline
\end{tabular}
\caption{Comparison of the exact answer and bounds for various $k$ and $s$}
\label{tab:bounds}
\end{table}
As observed from the data, the spectral bound is particularly accurate when $s \ll k$ holds. 

Conversely, when $k$ and $s$ are close, i.e., $s \approx k$, the Tur{\'a}n's bound provides a good approximation. Recall that Tur{\'a}n's's theorem gives
\[
\text{Tur{\'a}n's bound} = \binom{k}{2} - \frac{s+1}{2} \left \lfloor \frac{k}{s+1} \right \rfloor \left (\left \lfloor \frac{k}{s+1}\right \rfloor - 1 \right) = \frac{k(k-1)}{2},
\]
which is close to the exact maximum
\[
sk - \frac{s(s+1)}{2} \approx k\left(k-\frac{1}{2}\right) - \frac{k^2}{2} \approx \frac{k(k-1)}{2},
\]
since in this regime $k - \frac{s+1}{2} \approx \frac{k-1}{2}$ and $s \approx k$.

However, in the intermediate regime where both $k$ and $s$ are large but $k/s \gtrsim 2$,  neither bound provides an accurate approximation, as both the spectral and Tur{\'a}n's estimates deviate significantly from the exact maximum.

\section{Proof of the main result}

\begin{proof}[Proof of Theorem~\ref{main}]
Recall that we may assume $n/2 + 1, k \geq s + 2$. 
Let $\mathcal{S}$ be the family of all sets $S \subset \zn$ with $k$ elements that maximize the number of edges in the induced subgraph. A subset $P\subseteq \zn$ is called a \textit{block} if it consists of consecutive elements of $\mathbb{Z}/n\mathbb{Z}$, i.e $P=\{p_1,p_2,\ldots, p_r\}$, where $d(p_i, p_{i+1})=1$ for $i=1,2,\ldots, r-1$. For a set $U \in \mathcal{S}$ define $f(U)$ as the maximal number of elements of $U$ contained in a block of size $s$. For a subset $T\in \zn$ let $l_T, r_T\in T$ be the endpoints of $T$ such that $T$ is located clockwise from $l_T$ and counterclockwise starting from $r_T$. For a vertex $u\in \zn$ and a subset $U$ of $\zn$, such that $u \in U$, define $h(u,U)$ as the largest block consisting of elements of $U$ that contains $u$. Let $g(U)\subset \mathcal{P}(U)$ be the set of all subsets of $U$ with $f(U)$ elements that are contained in a block of size $s$.

Let $h(U):=\max_{T\in g(U)} h(l_T, T)$ and let $w(U)$ be a subset $T\in g(U)$ such that $h(U)=h(l_T,T)$.

We are now ready to choose an extremal element. Let $\mathcal{S}_1$ be the family of all sets $S\in \mathcal{S}$ that maximize $f(S)$. Let $\mathcal{S}_2$ be the family of all sets $S\in \mathcal{S}_1$ for which the quantity $h(S)$ is maximized. Pick any $M\in \mathcal{S}_2$ and let for simplicity $A:=w(M)$, $u:=l_A$, and $\mathcal{B}$ be a block of size $s$ that contains $A$ and let $\mathcal{A}\subset M$ be the block of size $h(M)$ that contains $u$.

Without loss of generality we may assume that $\mathcal{B}\subseteq \mathcal{A}$ or $\mathcal{A}\subseteq \mathcal{B}$ as otherwise we can just shift $A$ counterclockwise until we reach the end of $\mathcal{B}$ while preserving the maximality of $A$, so we may also assume that $u-1\notin M$. 

Assume for contradiction that the set $C:=\{w\in M\setminus A:d_{-}(w, M)>s\}$ is nonempty and let $v\in C$ be such that $d_{-}(v, M)$ is minimal. Consider the set $M':=M\setminus \{v\}\cup \{u-1\}$. Note that $u-1$ is adjacent to all vertices in $N^M_{+}(v)$ as all clockwise neighbors of $v$ lie in the set $\{w\in M\setminus A: d_{-}(w, M)\leq s\}$ by the minimality of $d_{-}(v, M)$ on $C$. We further have that $|N^M_{-}(v)|+1\leq |A|$ as $v$ and all of its counterclockwise neighbors lie in a block of size $s$ and so by maximality of $A$, it must have at least as many elements. Since $u-1$ is adjacent to at least $|A|-1$ elements from $\mathcal{B}$ (as $|\mathcal{B}|\leq s$, at most one vertex is $s+1$ units apart from $u-1$), which shows that replacing $v$ with $u-1$ gains $|A|-1-|N^M_{-}(v)|\geq 0$ edges in $M'$. However, the latter contradicts the maximality of $h(M)$ as adding $u-1$ increases $\mathcal{A}$

Thus, all $v \in M$ satisfy at least one of the following: 
\begin{enumerate}
    \item $v\in A$;
    \item $v \in \mathcal{A}$;
    \item $v\in N^M_{-}(u)$.
\end{enumerate}

We now show that $A\subseteq \mathcal{A}$. Assume for contradiction that $r_A\notin \mathcal{A}$. Then replace $r_{A}$ by $u-1$. We have that $u-1$ is adjacent to all the other vertices in $A$ (except maybe $r_{A}$, which we removed) and is adjacent to all other vertices in $M$ as all of them are on $N^M_{-}(u)$. Hence, the number of edges has not decreased, while $\mathcal{A}$ has increased, which again contradicts to the maximality of $h(M)$.  Thus, $A\subseteq \mathcal{A}$. 

Finally, if $v \in N^M_{-}(u)\setminus \mathcal{A}$, then note that we can replace $v$ by $u-1$ as all elements of $N^M_{-}(u)$ are neighbors of $u-1$ and $u-1$ has $\min(|\mathcal{A}|,s)$ neighbors from $\mathcal{A}$, which is at least the number of neighbors of $v$ in $\mathcal{A}$. This once again contradicts with the maximality of $\mathcal{A}$.  Hence $|\mathcal{A}|=k$ as required.

\end{proof}

\begin{corollary}
    If $k+s < n$, then we have 
    \[
    \max_{U\subset C_k^s} e(U) = sk - \frac{s(s+1)}{2}.
    \]
\end{corollary}

\begin{proof}
    Let $U:=\{1, 2, \ldots, k\}$. By the previous theorem it suffices to compute $e(U)$, which we will do by counting the degree of each vertex.\\
    \noindent \textit{The case $k\geq 2s$.} Each of the vertices $s+1, \ldots, k-s$ has $2s$ neighbors. The vertices $i\in \{1,2,\ldots, s\}$ and $k-i\in \{k-s+1,\ldots, k\}$ are of degree $s+i-1$. This gives a total of  
    \[
    \frac{2s(k-2s)+2\sum_{i=1}^s s+i-1}{2} = sk-\frac{s(s+1)}{2}
    \]
    edges. 

    \textit{The case $k<2s$.}  Here each of the vertices $i, k-i\in \{1, 2, \ldots, k-s\}$ is of degree $i-1+s$. Also each vertex $i$ for $i=k-s+1,k-s+2,\ldots, s$ has $i-1+k-i=k-1$ neighbors, which gives a total of 
    \[
    \frac{(2s-k)(k-1)+2\sum_{i=1}^{k-s} (i-1+s)}{2}=sk- \frac{s(s+1)}{2}
    \]
    edges.
\end{proof}

\bibliography{main}
\bibliographystyle{plain}

\end{document}